\documentclass{amsart}
\usepackage[utf8]{inputenc}
\usepackage{amsthm,amsmath,amssymb,mathdots}
\usepackage{fullpage}
\usepackage{graphics}
\usepackage{hyperref}
\usepackage[noadjust]{cite}
\usepackage{tikz}
\usetikzlibrary{cd}

\newtheorem{theorem}{Theorem}[section]
\newtheorem{lemma}[theorem]{Lemma}
\newtheorem{corollary}[theorem]{Corollary}
\newtheorem{proposition}[theorem]{Proposition}

\theoremstyle{definition}
\newtheorem{definition}[theorem]{Definition}

\newcommand{\R}{\mathbb{R}}

\newcommand{\C}{\mathbb{C}}

\newcommand{\boxb}{\square_b}
\renewcommand{\S}{S}

\newcommand{\Ell}{\mathcal{L}}
\newcommand{\vspan}{\mathrm{span}}
\newcommand{\dee}{\partial}

\newcommand{\abs}[1]{\left\lvert#1\right\rvert}

\newcommand{\E}{\mathcal{E}}
\newcommand{\levi}{\mathrm{Levi}}

\newcommand{\LL}{\mathbb{L}}
\newcommand{\harm}{\mathcal{H}}

\newcommand{\LLv}[1]{\mathcal{L}_{#1}}

\newcommand{\lt}{\mathcal{L}_t}
\newcommand{\bP}{\mathbb{P}}
\newcommand{\on}{\operatorname}
\renewcommand{\bar}{\overline}

\title{CR embeddability of quotients of the Rossi sphere via spectral theory}
\author{Henry Bosch}
\address[Henry Bosch]{Harvard University, Department of Mathematics, 
Cambridge, MA 02138, USA}
\email{henrybosch@college.harvard.edu }

\author{Tyler Gonzales}
\address[Tyler Gonzales]{Yale University, Applied Mathematics Program, 
New Haven, CT 06511, USA}
\email{tyler.gonzales@yale.edu}

\author{Kamryn Spinelli}
\address[Kamryn Spinelli]{Brandeis University, Department of Mathematics, Waltham, MA 02453, USA}
\email{kspinelli@brandeis.edu}

\author{Gabe Udell}
\address[Gabe Udell]{Cornell University, Department of Mathematics, 
Ithaca, NY 14853, USA}
\email{gru5@cornell.edu}

\author{Yunus E. Zeytuncu}
\address[Yunus E. Zeytuncu]{University of Michigan--Dearborn, Department of Mathematics and Statistics, 
Dearborn, MI 48128, USA}
\email{zeytuncu@umich.edu}

\thanks{This work is supported by NSF (DMS-1950102 and DMS-1659203). The work of the last author is also partially supported by a grant from the Simons Foundation (\#353525).}

\subjclass[2010]{Primary 32V05; Secondary 32V20}
\keywords{Kohn Laplacian, Rossi sphere, 3-manifolds}

\begin{document}

\begin{abstract}
We look at the action of finite subgroups of $\on{SU}(2)$ on $\S^3$, viewed as a CR manifold, both with the standard CR structure as the unit sphere in $\C^2$ and with a perturbed CR structure known as the Rossi sphere. We show that quotient manifolds from these actions are indeed CR manifolds, and relate the order of the subgroup of $\on{SU}(2)$ to the asymptotic distribution of the Kohn Laplacian's eigenvalues on the quotient. We show that the order of the subgroup determines whether the quotient of the Rossi sphere by the action of that subgroup is CR embeddable. Finally, in the unperturbed case, we prove that we can determine the size of the subgroup by using the point spectrum.
\end{abstract}

\maketitle

\section{Introduction}
Every strongly pseudoconvex CR manifold with real dimension at least 5 is globally embeddable (see \cite{Monvel75} and \cite[Chapter 12]{CS01}).  Thus, 3-dimensional strongly pseudoconvex CR manifolds are especially interesting for their potential non-embeddability.  As early as 1965, \cite{Rossi65} described a perturbed CR structure on $S^3\subset \C^2$ that was not globally embeddable into any $\mathbb{C}^N$. The resulting abstract CR manifold is called the Rossi Sphere. One can see the obstruction to embeddability by noting that all the CR functions in this setting are even (see \cite[Chapter 12]{CS01}). Another formulation of the non-embeddability is the existence of arbitrarily small eigenvalues in the spectrum of the Kohn Laplacian (which is equivalent to the non-closedness of the range of $\overline{\partial}_b$) defined by the perturbed CR structure in this setting (see \cite{REU17, REU18}). Indeed, one can show that the perturbed Kohn Laplacian has eigenvalues that converge to $0$ from the right hand side. In other words, $0$ belongs to the essential spectrum of the Kohn Laplacian. 

In this note, we follow the spectral-theoretic approach of \cite{REU17, REU18} and investigate the CR embeddability of the quotients of the Rossi sphere by the action of finite subgroups of $\on{SU}(2)$. First, we focus on the quotient of the Rossi sphere by the action of the antipodal map, which is topologically the real projective space $\R \bP^3$. This quotient is known to be CR embeddable into $\mathbb{C}^3$; we present a new proof of this statement by showing that the non-zero eigenvalues of the Kohn Laplacian here are uniformly bounded away from zero. Later, in Section \ref{sec:rossi-embeddability-classification}, we consider quotients by other finite subgroups of $\on{SU}(2)$ and relate the embeddability question to the order of the subgroup: specifically, the quotient by a finite subgroup of $\on{SU}(2)$ is CR embeddable if and only if the order of the group is even. In Section \ref{sec:hear-group-order}, we pursue the relation between the quotient of a finite group $G\subseteq \on{SU}(2)$ with the eigenvalues of the unperturbed Kohn Laplacian on the quotient of $\S^3$ by $G$.
 We show that one can hear the size of the quotienting group on the sphere (following the terminology of \cite{Kac}, we say that we can hear a quantity, or it is audible, if the quantity may be explicitly determined given the spectrum). We note that similar results on Lens spaces in the Riemannian setting were observed in \cite{Ikeda79}.

We refer the reader to standard books \cite{CS01} and \cite{Boggess91CR} for the basics of CR geometry. For completeness, we review some of these points on quotient manifolds in the second section. We use the spectral theory approach to the embedding problem utilized recently in \cite{REU17, REU18} that dates back to the work in \cite{Burns79, BurnsEpstein}. We also note the work of Fu in \cite{Fu2005, Fu2008} for further connections between geometry and spectral theory of the $\overline{\partial}$-Neumann Laplacian. 
 
In the rest of this note, we use $\S^3$ to denote the unit sphere in $\mathbb{C}^2$. We use the following notation for two vector fields $\LLv{}$ and $\overline{\LLv{}}$ on $S^3 \subset \C^2$:
\begin{align*}
\LLv{} = \overline{z}_1 \frac{\partial}{\partial z_2} - \overline{z}_2 \frac{\partial}{\partial z_1} \text{ and }
\overline{\LLv{}} = z_1 \frac{\partial}{\partial \overline{z}_2} - z_2 \frac{\partial}{\partial \overline{z}_1}.
\end{align*}
The standard CR structure on $\S^3$ is generated by the $(1,0)$ vector field $\LLv{}$. The perturbed CR structure on $\S^3$ is generated by $\lt = \LLv{} + \overline{t} \overline{\LLv{}}$ for complex $|t| < 1$. The pair $(\S^3, \lt)$ denotes the Rossi sphere. We use $\boxb$ to denote the Kohn Laplacian on $(\S^3,\LLv{})$ and $\boxb^t$ to denote the one on $(\S^3,\lt)$. A simple calculation gives the following representation
\[\boxb^t=-\lt \frac{1 + |t|^2}{(1 - |t|^2)^2} \overline{\lt}~ \text{ for any }|t|<1.\]


\section{CR geometry and discrete quotients}

We start with introducing the notation. Given a smooth map $f: M \to N$ of manifolds, we denote its differential by $df$. Its action on the tangent space at $p \in M$ is $d_p f$. Given a vector space or vector bundle $V$, we let $V^\C$ be its complexification; in the case of a tangent space $T_p M$ or tangent bundle $TM$, the complexification is denoted   $T_p^\C M$ and $T^\C M$, respectively. Given a map $\phi$ of vector spaces or vector bundles, its complexification (given by extending by linearity over complex scalars), is given by $\phi^\C$. 

Given a discrete group action on a CR manifold $M$, the CR structure of $M$ may be passed through to the quotient in a natural way when the group action in some sense agrees with the CR bundle. 
$M$ and its quotient are locally indistinguishable, and therefore the CR structure descends to the quotient in a natural way.

\subsection{Induced CR structures on quotient manifolds}

 The definition below details what it means for a CR bundle to be \textit{compatible} with the action of a discrete group. Here and throughout, we denote by $\ell_g$ the map $M \to M$ given by action of the group element $g$. 
 
\begin{definition}
    Let $M$ be a smooth manifold, and $T^{\C}M = TM \otimes \C$ its complexified tangent bundle. Let $\Gamma$ be a group which acts smoothly on $M$. We say that a subbundle $\LL$ of $T^{\C}M$ is \textit{compatible} with $\Gamma$ if for each $g \in \Gamma$,  $(d\ell_{g})^{\C}$ restricts to a bundle isomorphism $\LL \to \LL$. 
\end{definition}
Note that when $M$ is a Lie group and $\LL$ a vector bundle which is invariant with respect to complexified differentials of left-multiplications, $\LL$ is automatically compatible with the action of any discrete subgroup of $M$. When the CR bundle of $M$ is compatible with $\Gamma$, it gives rise to a CR structure on the quotient manifold $M / \Gamma$.


\begin{proposition}\label{quotient}
    Let $M$ be an abstract CR manifold with defining bundle $\LL \subset T^{\C}M$. Let $\Gamma$ be a discrete group which acts smoothly, freely and properly on $M$, and suppose that $\LL$ is compatible with $\Gamma$. Then the orbit space $M / \Gamma$ is a topological manifold of dimension equal to $\dim M$ and has a unique smooth structure such that the quotient map $\pi: M \to M / \Gamma$ is a local diffeomorphism. Further, $M / \Gamma$ is an abstract CR manifold with defining bundle $\LL / \Gamma = (d\pi)^{\C}(\LL)$.
\end{proposition}


\subsection{The induced Cauchy-Riemann complex and Kohn Laplacian on quotient manifolds}
In order to define the Kohn Laplacian on a quotient manifold $M / \Gamma$, we need to deduce the structure of its tangential Cauchy-Riemann complex, see also \cite[Chapter 8]{Boggess91CR}. Indeed, for quotients defined as above, the tangential Cauchy-Riemann complex on the quotient manifold is induced by that on the parent manifold.

\begin{proposition}
    Suppose that $M$ satisfies the hypotheses of Proposition \ref{quotient}, and that $M$ has a Riemannian metric such that the action of every $g \in \Gamma$ is an isometry.
    Then in addition to the conclusions of Proposition \ref{quotient}, the intrinsic tangential Cauchy Riemann complex on $M/\Gamma$ is induced by that on $M$. More precisely, let $\E_M(p,q)$ and $\E_{M/\Gamma}(p,q)$ be the fiber bundles of $(p,q)$-forms on $M$ and $M/\Gamma$, respectively, and let $\bar{\dee}_M$ and $\bar{\dee}_{M / \Gamma}$ be the tangential Cauchy-Riemann operators associated with the CR structures on $M$ and $M / \Gamma$, respectively. Also let $\phi: M \to M / \Gamma$ be the quotient map. Then for each $p,q \geq 0$, $\overline{\dee}_{M} \circ \phi^* = \phi^* \circ \overline{\dee}_{M/\Gamma}$ as maps $\E_{M/\Gamma}^{p,q} \to \E_{M}^{p,q+1}$. In other words, the following diagram commutes:

    \tikzcdset{column sep/superlong/.initial=6em}
    \[\begin{tikzcd}[row sep=huge, column sep=superlong]
        \E_M^{p,q} \arrow[r,"\overline{\dee}_{M}"]
            & \E_M^{p,q+1} \\
        \E_{M/\Gamma}^{p,q} \arrow[r, "\overline{\dee}_{M/\Gamma}"] \arrow[u, "\phi^*"]
            & \E_{M/\Gamma}^{p,q+1} \arrow[u, "\phi^*"]
    \end{tikzcd}\]
\end{proposition} 

The same statement is true for the adjoint operator $\bar \dee^*$ by standard arguments. It follows that the Kohn Laplacian on $M / \Gamma$ is induced by that on $M$. Because the Kohn Laplacian acts locally, we can compute its action on $\Gamma$-invariant forms on $M$ by computing the Kohn Laplacian on $M / \Gamma$ and then pulling back to $M$, or by first pulling back to $M$ and then computing the Kohn Laplacian there.

\begin{proposition} \label{kohn-laplacian-commute}
    Let $\boxb^M$ and $\boxb^{M / \Gamma}$ denote the Kohn Laplacian on $M$ and $M / \Gamma$, respectively. Then $\phi^* \circ \boxb^{M / \Gamma} = \boxb^M \circ \phi^*$.
\end{proposition}
\begin{proof}
    The statement follows from the previous proposition regarding $\bar{\dee}$ and the analogous statement for $\bar{\dee}^*$.
\end{proof}
For readers interested in reading the details of these arguments on CR manifolds, we recommend \cite[Chapter 4]{CS01}. For the similar ideas on domains (not on manifolds), we recommend \cite{StraubeBook}.

\subsection{Equivalence of the Levi form on quotient CR manifolds; strong pseudoconvexity and quotients}
The final piece necessary to relate the embeddability to the spectrum of the Kohn Laplacian is strong pseudoconvexity. The same idea applies here as in the previous arguments: since strong pseudoconvexity is a local property, and we are taking quotients by discrete groups acting isometrically, the identification of local geometry on $M$ and $M / \Gamma$ allows us to say that $M / \Gamma$ is strongly pseudoconvex as well. In fact, the condition of strong pseudoconvexity is intimately tied to the Levi form on a CR manifold, which can be clearly seen to act locally.

\begin{proposition}
Let $M$ be a strongly pseudoconvex CR manifold of hypersurface type. Then any discrete quotient of $M$ in the way defined above is strongly pseudoconvex.
\end{proposition}
\begin{proof}
    The claim follows from considering the Levi form on $M / \Gamma$ and observing that it is exactly the pushforward by the quotient map of the Levi form on $M$.
\end{proof}

For more details, again see \cite[Chapter 10]{Boggess91CR}. As an application of this machinery, we gain information about quotients of the sphere by certain subgroups. 
\begin{corollary} \label{rossi-quotients-are-cr-mflds}
    Let $\Gamma$ be a finite subgroup of $\on{SU}(2)$.
    Then $\Gamma$ acts on $S^3 \cong \on{SU}(2)$ by left multiplication.
    The quotient of the sphere by $\Gamma$, endowed with either the standard or Rossi CR structure, is a strongly pseudoconvex abstract CR manifold.
\end{corollary}
\begin{proof}
    The standard CR vector fields $\Ell, \bar{\Ell}$ are invariant with respect to complexified differentials of left-multiplications in the Lie group and therefore are compatible with the group action. The perturbed CR vector fields $\Ell_t, \bar{\Ell_t}$ are also left-invariant in this way, because they are linear combinations of $\Ell$ and $\bar{\Ell}$; hence, the compatibility condition of $\Ell_t$ and $\bar{\Ell_t}$ is automatically satisfied. Therefore, $\S^3 / \Gamma$ is a well-defined abstract CR manifold with this perturbed CR structure. As $\S^3$ is strongly pseudoconvex when endowed with either CR structure, $\S^3 / \Gamma$ is too.
\end{proof}


\section{Embeddability of $(\mathbb{RP}^3, \overline{\mathcal{L}_t})$ via spectral theory}\label{three}

In this section, we prove that the quotient of the Rossi sphere by the antipodal map is an embeddable CR manifold via spectral theory (see also \cite{Burns79} and \cite[Chapter 12]{CS01}).  We denote this CR manifold by $(\mathbb{RP}^3, \overline{\mathcal{L}_t})$ as in \cite{Epstein}

Since $\on{SU}(2)\cong \S^3 \subset \C^2$ as a smooth manifold, in this part we write elements of $\on{SU}(2)$ as points $(\xi_1,\xi_2)\in \C^2$ with $\abs{\xi_1}^2+\abs{\xi_2}^2=1$ and where $(\xi_1,\xi_2)\cdot (z_1,z_2)=(\xi_1 z_1-\overline{\xi_2}z_2,\xi_2 z_1 +\overline{\xi_1}z_2)$. In this section, we consider the quotient of the 3-sphere $\S^3$ by the antipodal map, that is the left action of the embedded subgroup $C_2 = \{(1, 0), (-1, 0)\} = \{\pm \mathrm{Id}\}$, endowed with the perturbed CR structure. Thanks to Corollary \ref{rossi-quotients-are-cr-mflds}, this quotient is a well-defined, strongly pseudoconvex, abstract CR manifold. Moreover, the Rossi sphere and its quotient have the same local geometry. This allows us to carry out our study of the spectrum of $\boxb^t$ on the quotient by investigating the spectrum of $\boxb^t$ as it acts on $C_2$-invariant functions on the sphere. 

\subsection{Spherical harmonic decomposition; invariant subspaces under $\boxb^t$}
The space of $L^2$-integrable functions on $\S^3$ has a convenient decomposition into spaces of homogeneous spherical harmonics. Let $\harm_k(\S^3)$ denote the space of degree-$k$ spherical harmonics and $\harm_{p,q}(\S^3)$ denote the space of bidegree-$(p,q)$ spherical harmonics\footnote{We drop $\S^3$ in the notation when the context is clear.}. We have the following standard decomposition:
    \[L^2(\S^3) = \bigoplus_{k=0}^\infty \harm_{k}(\S^3).\]
We refer to \cite{Axler13Harmonic} for further details on spherical harmonics. The following statement about $L^2$-integrable functions on the quotient is evident.
\begin{proposition}
    $L^2(\S^3 / C_2)$ is canonically isomorphic to the set of functions in $L^2(\S^3)$ which are invariant under the action of $C_2$, that is even functions $f$ where $f(z) = f(-z)$ on the sphere.
\end{proposition}
Note that the homogeneous spherical harmonics which are also even functions are exactly those with even degree. Putting the previous two propositions together produces a basis for the $L^2$ space on the antipodal quotient. Here is the updated decomposition in this setting:
    \[L^2(\S^3 / C_2) \cong \bigoplus_{k=0}^\infty \harm_{2k}(\S^3).\]

Much of the following work is similar to the ideas in \cite{REU17}. However, \cite{REU17} examined only the subspaces $\harm_{2k+1}$ corresponding to spherical harmonics of odd degree. Hence, we develop corresponding results for the subspaces of even-degree spherical harmonics. The following proposition tells us about the action of part of $\boxb^t$ on spherical harmonics. A more general statement dates back to \cite{Folland}, and a more explicit computation can be found in \cite{REU18}.

\begin{proposition}
\label{action-l-lbar}
    Let $f\in \harm_{p,q}(\S^3)$. Then
    \[\boxb f = - \Ell \bar{\Ell} f = 2 (pq+q) f = 2 q(p+1) f \quad \text{and} \quad -\bar{\Ell} \Ell f = 2 (pq+p) f = 2 p(q+1) f.\]
\end{proposition}

For $f\in\harm_{0,2k}(\S^3)$, we define the following two spaces: $$V_{f} = \vspan\{f, \bar{\Ell}^2 f, \dots, \bar{\Ell}^{2k} f\}\;\;\textup{(a $k+1$-dimensional space)},$$ and $$W_{f} = \vspan\{\bar{\Ell} f, \dots, \bar{\Ell}^{2k-1} f\}\;\;\textup{(a $k$-dimensional space)}.$$ 
Given an orthogonal basis $\{f_0, \dots, f_{2k}\}$ of $\harm_{0,2k}$, $\harm_{2k}$ decomposes as
\[\harm_{2k} = \bigoplus_{i=0}^{2k} V_{f_i} \oplus W_{f_i}.\]

For the moment we ignore the constant $h(t)$ in $\boxb^t = -h(t) (\Ell \bar\Ell + |t|^2 \bar\Ell \Ell + t\Ell^2 + \bar t \bar \Ell^2)$, where $h(t) = \frac{1 + |t|^2}{(1 - |t|^2)^2}$.

\begin{proposition}
The subspaces $V_f$ and $W_f$ are invariant under $\boxb^t$.
\end{proposition}
\begin{proof}
    For $0\leq\sigma\leq 2k$, let $v_\sigma = \bar\Ell^\sigma f$ be a basis element of $V_{f}$ or $W_{f}$. We compute the action of each piece of $\boxb^t$ on $v_\sigma$, using the fact that $v_\sigma \in \harm_{\sigma, 2k-\sigma}(\S^3)$ and the identities in the Proposition \ref{action-l-lbar}. The action of the $\Ell^2$ piece is most sophisticated:
    \begin{align*}
        \Ell^2 v_\sigma &= \Ell^2 \bar \Ell^\sigma f = \Ell (\Ell \bar\Ell (\bar\Ell^{\sigma-1} f)) = -2 (2k-\sigma+1)(\sigma) \Ell (\bar\Ell^{\sigma-1} f_i) = -2 (2k-\sigma+1)(\sigma) \Ell \bar\Ell (\bar\Ell^{\sigma-2} f) \\ 
        &= 4 (2k-\sigma+1)(\sigma) (2k-\sigma+2)(\sigma-1) \bar\Ell^{\sigma-2} f = 4 (2k-\sigma+1)(\sigma) (2k-\sigma+2)(\sigma-1) 
        v_{\sigma-2}.
    \end{align*}
    For the other pieces, the computation is simpler:
    \begin{align*}
        \bar\Ell^2 v_\sigma &= \bar\Ell^2 \bar\Ell^\sigma f = \bar\Ell^{\sigma+2} f = v_{\sigma+2}, \\
        \Ell \bar\Ell v_\sigma &= -2 (2k-\sigma)(\sigma+1) v_\sigma, \\
        \bar\Ell \Ell v_\sigma &= -2 (\sigma)(2k-\sigma+1) v_\sigma.
    \end{align*}
    In these computations we take $v_\tau$ to be $0$ if $\tau$ is not between $0$ and $2k$ (this is consistent with the definitions of $\Ell$ and $\bar\Ell$).
\end{proof}
In fact, the calculations in the preceding proposition's proof provide us the following matrix representations of the action of $\boxb^t$ on the finite dimensional subspaces $V_f$ and $W_f$.
\begin{corollary}
    The matrix representation of $\boxb^t = - \Ell_t \bar{\Ell}_t$ on the subspace $V_f$ is
    \[\begin{pmatrix}
        d_1 & -u_1 & & \\
        -\bar{t} & d_2 & -u_2 & \\
         & -\bar{t} & \ddots  & -u_k \\
         & & -\bar{t} & d_{k+1}
    \end{pmatrix},\]
    where 
    \[u_j = 4 t (2j) (2j-1) (2k-2j+1) (2k-2j+2) \text{ and }
    d_j = 2 ( (2k-2j+2) (2j-1) + |t|^2 (2j-2) (2k-2j+3) ).\]
    Similarly, the matrix representation of $\boxb^t$ on the subspace $W_f$ is
    \[\begin{pmatrix}
        d_1 & -u_1 & & \\
        -\bar{t} & d_2 & -u_2 & \\
         & -\bar{t} & \ddots & -u_{k-1} \\
         & & -\bar{t} & d_{k}
    \end{pmatrix},\]
    where 
    \[u_j = 4 t (2j+1) (2j) (2k-2j) (2k-2j+1) \text{ and } d_j = 2 ( (2k-2j+1) (2j) + |t|^2 (2j-1) (2k-2j+2) ).\]
\end{corollary}

Note that in the $V_f$ matrices, the product of the $j$-th corresponding off-diagonal entries is 
\[|t|^2 (2j) (2j-1) (2k-2j+1) (2k-2j+2) > 0,\]
so we can apply Proposition 5.1 from \cite{REU17} to get the following simpler representations.

\begin{proposition}
    Let $A$ be the matrix representation of $\boxb^t$ on $V_{f}$. Then $A = S^{-1} B S$ where $S$ is a diagonal matrix and 
    \[B = \begin{pmatrix}
        d_1 & \sqrt{u_1 \bar{t}} & & \\
        \sqrt{u_1 \bar{t}} & d_2 & \sqrt{u_2 \bar{t}} & \\
         & \sqrt{u_2 \bar{t}} & \ddots & \sqrt{u_k\bar{t}} \\
         & & \sqrt{u_k \bar{t}} & d_{k+1}
    \end{pmatrix}.\]
    Note that the quantities in the square roots are real and positive, so there is no ambiguity.
    There is a factor of $2$ in each entry of the matrix $B$; to simplify our calculations we omit this factor (doing so does not interfere with the positivity of the spectrum). 
      After disregarding the factor of $2$, the $j$-th off-diagonal entry in $B$ is
    \[\sqrt{(2j) (2j-1) (2k-2j+1) (2k-2j+2) t \bar{t}} = |t| \sqrt{(2j) (2j-1) (2k-2j+1) (2k-2j+2)}.\]
\end{proposition}

\subsection{Equality of eigenvalues of $\boxb^t$ on the two types of invariant subspaces} In this subsection, 
we show that for spherical harmonics of a fixed even degree, all the eigenvalues of $\boxb^t$ on $W_f$ are also eigenvalues of $\boxb^t$ on $V_f$ (see Theorem \ref{matching eigenvalues}). Since $\dim V_f = \dim W_f + 1$, this shows that the action of $\boxb^t$ on $V_f$ is exactly the action on $W_f$, with an extra one-dimensional subspace mapping identically to $0$. 

Let $F$ denote the linear function on matrices which flips a matrix's entries across its antidiagonal. For example,
\[F \begin{pmatrix}
    a & b \\
    c & d
\end{pmatrix} = \begin{pmatrix}
    d & b \\
    c & a
\end{pmatrix}.\]
We can express this operator in terms of matrix multiplications as follows.
\begin{lemma}
    Let $M \in \mathbb{C}^{m \times m}$ and 
    $$E = \begin{pmatrix}& & 1 \\ & \iddots & \\ 1 & & \end{pmatrix} \in \mathbb{C}^{m \times m}.$$Then $F(M) = EM^TE$.
\end{lemma}
\begin{proof}
    Consider the $i,j$ entry of  both $F(M)$ and $EM^TE$. Observe that $F(M)_{i,j} = M_{m+1-j,m+1-i}.$ To see the $i,j$ entry of $EM^TE$, write
    $$(EM^TE)_{ij} = \sum_{k=1}^n E_{ik} (M^T E)_{kj} = \sum_{k=1}^n E_{ik} \sum_{l = 1}^n (M^T)_{kl} E_{lj} = \sum_{l,k=1}^n E_{ik} M_{lk} E_{jl}$$
    Note that $E_{st} = \delta_{s + t - m - 1}$. Therefore, the only nonzero term in the above sum has $k = m+1-i$ and $l = m+1-j$, so $(EM^TE)_{ij} = M_{m+1-j,m+1-i}$.
\end{proof}

\begin{corollary}
  Let $M \in \mathbb{C}^{m \times m}$. Then $M$ and $F(M)$ are similar. 
\end{corollary}
\begin{proof}
    Since $E = E^{-1}$, the above shows that $F(M)$ is similar to $M^T$, which is similar to $M$ since $\mathbb{C}$ is algebraically closed \cite[Theorem 11.8.1]{garcia_horn_2017}.
\end{proof}

\begin{lemma}[\cite{garcia_horn_2017}, Theorem 9.7.2]
    If $A\in \C^{m\times n}$ and $B\in \C^{n\times m}$, then the nonzero eigenvalues of $AB\in \C^{m \times m}$ and $BA\in \C^{n \times n}$ are the same with identical algebraic multiplicity.
\end{lemma}

Let $L_V=\mathcal{L}\mid_{V_f}$, $L_W=\mathcal{L}\mid_{W_f}$, $\bar{L}_V=\bar{\mathcal{L}}\mid_{V_f}$, and $\bar{L}_W=\bar{\mathcal{L}}\mid_{W_f}$. Then $\boxb\mid_{V_f} = -L_W \bar{L}_V$ and $\overline{\boxb}\mid_{V_f} = -\bar{L}_W L_V$. 

\begin{lemma}
    For the matrix representations in the bases we have been using, $F[L_W \bar{L}_V] = \bar{L}_W L_V$ and $F[\bar{L}_W L_V]=L_W \bar{L}_V$.
\end{lemma}
\begin{proof}
    First, recall that (after dropping the universal constant 2) for $f\in \mathcal{H}_{p,q}$, $\boxb f = (pq+q)f$ so $\overline{\boxb} f=\overline{\overline{\overline{\boxb} f}} =\overline{\boxb \overline{f}}
     =\overline{(qp+p)\overline{f}}=(qp+p)f$. 
    Notice that this means that on both $V_f$ and $W_f$, every basis element is an eigenvector of both $\boxb$ and $\overline{\boxb}$ and so each are represented by diagonal matrices in this basis. On $V_f$ the $j$'th diagonal element of the matrix representation of $\boxb$ is $(2k-2j+2)(2j-1)$ and the $j$'th diagonal element of $\overline{\boxb}$ is $(2j-2)(2k-2j+3)$. Thus, $F[\boxb\mid_{V_f}]$ is diagonal with the $j$th diagonal element equal to $$(2k-2(k+2-j)+2)(2(k+2-j)-1)=(2j-2)(2k-2j+3).$$ Therefore, $\boxb$ on $V_f$ flips to being $\overline{\boxb}$ on $V_f$ and so $F[L_W\overline{L}_V]=\overline{L}_W L_V$. Since $F$ is an involutory operation this also tells us that $F[\overline{L}_W L_V]=L_W\overline{L}_V$.
\end{proof}

\begin{lemma}
    For the matrix representations in the bases we have been using, $F[\bar{L}_W\bar{L}_V]=\bar{L}_W\bar{L}_V$ and $F[L_WL_V]=L_WL_V$.
\end{lemma}
\begin{proof}
    This follows from the matrix representations which we already have for $\bar{\mathcal{L}}^2\mid_{V_f} = \bar{L}_W\bar{L}_V$ and $\mathcal{L}^2\mid_{W_f}=L_WL_V$ which are used to form the subdiagonal and superdiagonal of the matrix representation of $\boxb^t$.
\end{proof}

\begin{theorem}\label{matching eigenvalues}
    $\boxb^t\mid_{V_f}$ and $\boxb^t\mid_{W_f}$ have the same nonzero eigenvalues with multiplicity.
\end{theorem}
\begin{proof}
    We apply our knowledge about the operator $F$ to deduce
    \begin{align*}
    \boxb^t\mid_V&\sim
        F[\boxb^t\mid_V]\\
        &=F[(\bar{t}\bar{L}_W+L_W)(tL_V+\bar{L}_V)]\\
        &= \lvert t\rvert^2 F[\bar{L}_W L_V] +F[L_W\bar{L}_V]+\bar{t} F[\bar{L}_W\bar{L}_V] +t F[L_W L_V]\\
        &=\lvert t \rvert^2 L_W \bar{L}_V +\bar{L}_W L_V +\bar{t} \bar{L}_W\bar{L}_V +t L_W L_V \\
        &= (\bar{L}_W+tL_W)(L_V+\bar{t}\bar{L}_V)\\
        &=\bar{\mathcal{L}_t}\mid_{W} \mathcal{L}_t\mid_{V}
    \end{align*}
    Since $\bar{\mathcal{L}_t}\mid_{W_f} \mathcal{L}_t\mid_{V_f}$ has the same nonzero eigenvalues as $\mathcal{L}_t\mid_{V_f}\bar{\mathcal{L}_t}\mid_{W_f} = \boxb^t\mid_{W_f}$, $\boxb^t\mid_{V_f}$ and $\boxb^t\mid_{W_f}$ have the same nonzero eigenvalues (note that we do not distinguish between algebraic and geometric multiplicity here, since the operators are self-adjoint).
\end{proof}

\subsection{Bounds on nonzero eigenvalues; embeddability}
We aim to bound the spectrum of $\boxb^t$ (acting on the orthogonal complement of its kernel) away from zero. Since all the nonzero eigenvalues of $\boxb^t$ on $V_f$ are also eigenvalues of $\boxb^t$ on $W_f$, it suffices to bound the eigenvalues of $\boxb^t$ on the $W_f$ spaces away from zero. Our main tool is the Gershgorin circle theorem (see \cite{gersh}).

\begin{theorem}[Gershgorin]
    Let $A$ be a complex $n \times n$ matrix over $\C$. For each $i = 1, \dots, n$, let $R_i = \sum_{j \neq i} |a_{ij}|$ and let $D(a_{ii}, R_i)$ denote the disk in the complex plane centered at $a_{ii}$ with radius $R_i$. Then every eigenvalue of $A$ is contained in the union of these Gershgorin disks:
    \[\bigcup_{i=1}^n D(a_{ii}, R_i).\]
\end{theorem}

We next represent the action of $\boxb^t$ on $W_f$ by a symmetric matrix with real coefficients. Such a matrix necessarily has all real eigenvalues, so the Gershgorin disks in this case degenerate to intervals in the real line. Recall that earlier we wrote the matrix representation of $\boxb^t$ on $W$ subspaces as
\[\begin{pmatrix}
    d_1 & -u_1 & & \\
    -\bar{t} & d_2 & -u_2 & \\
     & -\bar{t} & \ddots & -u_{k-1} \\
     & & -\bar{t} & d_{k}
\end{pmatrix},\]
where $u_j = 4 t (2j+1) (2j) (2k-2j) (2k-2j+1) \text{ and } d_j = 2 ( (2k-2j+1) (2j) + |t|^2 (2j-1) (2k-2j+2) ).$
Just as we did with the matrix representation on $V_f$, we can transform this matrix into a symmetric matrix:
\[M = \begin{pmatrix}
    d_1 & \sqrt{u_1 \bar{t}} & & \\
    \sqrt{u_1 \bar{t}} & d_2 & \sqrt{u_2 \bar{t}} & \\
     & \sqrt{u_2 \bar{t}} & \ddots & \sqrt{u_{k-1} \bar{t}} \\
     & & \sqrt{u_{k-1} \bar{t}} & d_k
\end{pmatrix}.\] As in the matrix representation for $\boxb^t$ on $V_f$, each entry in $M$ has a factor of $2$; just like before, we disregard this factor for simplicity. What remains is the matrix
\[M = \begin{pmatrix}
    m_1 & s_1 & & \\
    s_1 & m_2 & s_2 & \\
     & s_2 & \ddots & s_{k-1} \\
     & & s_{k-1} & m_k
\end{pmatrix}\]
where $m_j = (2k-2j+1) (2j) + |t|^2 (2j-1) (2k-2j+2) \text{ and } s_j = |t| \sqrt{(2j+1) (2j) (2k-2j+1) (2k-2j)}.$ We proceed with the Gershgorin analysis.
\begin{proposition} \label{w-eigenvalues-bounded-below}
    Let $k\geq 4$, and let $M$ be as above. Then the eigenvalues of $M$ are bounded below by $1$.
\end{proposition}
\begin{proof}
    The Gershgorin intervals for each row of the matrix are as follows.
    \begin{align*}
        D_1 &: (m_1 - s_1, m_1 + s_1) \\
        D_j &: (m_j - s_{j-1} - s_j, m_j + s_{j-1} + s_j) \qquad (2\leq j\leq k-1) \\
        D_k &: (m_k - s_{k-1}, m_k + s_{k-1})
    \end{align*}
    We need to show that the lower limit on each of these intervals is at least $1$. We treat each case separately.
        \textbf{Case One $D_1$:} In this case we have the following lower bound on $m_1 - s_1$:
        \begin{align*}
            m_1 - s_1 &= (2k-2+1)(2) + |t|^2 (1)(2k-2+2) - |t|\sqrt{(3)(2)(2k-2+1)(2k-2)} \\
            &= (2k-1)(2) + |t|^2 (2k) - |t|\sqrt{6(2k-1)(2k-2)} \geq (2k-1)(2) + |t|^2 (2k) - |t| \sqrt{6} (2k-1)
        \end{align*}
        This is minimized over $|t|$ when $|t| = \frac{\sqrt{6}(2k-1)}{2(2k)} \geq 1$. As we only consider $|t|\in(0,1)$, it suffices to consider what happens when $|t|=1$. When $|t|=1$, we have
        \[m_1 - s_1 \geq (2k-1)(2) + (2k) - \sqrt{6} (2k-1) = (2-\sqrt{6})(2k-1) + (2k) > (2k) - (2k-1) = 1.\]
        \textbf{Case Two $D_k$:} In this case we have the following lower bound on $m_k - s_{k-1}$:
        \begin{align*}
            m_k - s_{k-1} &= (2k-2k+1)(2k) + |t|^2 (2k-1)(2k-2k+2) \\
            &\qquad - |t|\sqrt{(2(k-1)+1) (2(k-1)) (2k-2(k-1)+1) (2k-2(k-1))} \\
            &= (2k) + |t|^2 (2k-1)(2) - |t|\sqrt{(2k-1)(2k-2)(3)(2)} \\
            &\geq (2k) + |t|^2 (2)(2k-1) - |t| \sqrt{6}(2k-1)
        \end{align*}
        This is minimized over $|t|$ when $|t| = \frac{\sqrt{6} (2k-1)}{2(2)(2k-1)} = \frac{\sqrt{6}}{4}$. For this value of $|t|$, we have
        \begin{align*}
            m_k - s_{k-1} &\geq (2k) + \frac{6}{16} (2)(2k-1) - \frac{\sqrt{6}}{4} \sqrt{6}(2k-1) \\
            &= (2k) + \frac{3}{4}(2k-1) - \frac{3}{2}(2k-1) 
            = (2k) - \frac{3}{4}(2k-1)
            \geq (2k) - (2k-1) = 1
        \end{align*}
        \textbf{Case Three $D_j$, $2\leq j\leq k-1$:} This is the most technical of the three cases. We have the following lower bound on $m_j - s_{j-1} - s_j$:
        \begin{align*}
            m_j - s_{j-1} - s_j &= (2k-2j+1)(2j) + |t|^2 (2j-1)(2k-2j+2) \\
            &\qquad - |t| \sqrt{(2(j-1)+1) (2(j-1)) (2k-2(j-1)+1) (2k-2(j-1))} \\
            &\qquad - |t| \sqrt{(2j+1) (2j) (2k-2j+1) (2k-2j)} \\
            &= (2k-2j+1)(2j) + |t|^2 (2j-1)(2k-2j+2) \\
            &\qquad - |t| \sqrt{(2j-1) (2j-2) (2k-2j+3) (2k-2j+2)} \\
            &\qquad - |t| \sqrt{(2j+1) (2j) (2k-2j+1) (2k-2j)}.
        \end{align*}
        This is minimized when
        \[|t| = \frac{r_1 + r_2}{2(2j-1)(2k-2j+2)}\]
        where for convenience, we are temporarily using the symbols
        \[r_1 = \sqrt{(2j-1) (2j-2) (2k-2j+3) (2k-2j+2)} \text{ and } r_2 = \sqrt{(2j+1) (2j) (2k-2j+1) (2k-2j)}.\]
        For this value of $|t|$, we have
        \begin{align*}
            m_j - s_{j-1} - s_j &\geq (2k-2j+1)(2j) + \frac{(r_1 + r_2)^2}{4(2j-1)(2k-2j+2)} - \frac{(r_1 + r_2)^2}{2(2j-1)(2k-2j+2)} \\
            &= (2k-2j+1)(2j) - \frac{(r_1 + r_2)^2}{4(2j-1)(2k-2j+2)}
            = (2k-2j+1)(2j) - \frac{r_1^2 + r_2^2 + 2r_1r_2}{4(2j-1)(2k-2j+2)}.
        \end{align*}
        Set $a = 2j$, $b = 2k - 2j$. The previous constraints on $j$ and $k$ translate to the new constraints $a+b = 2k$, $4 \leq a \leq 2k-2$, $2 \leq b \leq 2k-4$. The above inequality becomes
        \begin{align*}
            m_j - s_{j-1} - s_j &\geq \underbrace{(b+1)(a) - \frac{(a-1)(a-2)(b+3)(b+2) + (a+1)(a)(b+1)(b)}{4(a-1)(b+2)}}_{p_1(a,b)} \\
            &\qquad - \underbrace{\frac{2\sqrt{(a-1)(a-2)(b+3)(b+2) (a+1)(a)(b+1)(b)}}{4(a-1)(b+2)}}_{p_2(a,b)}.
        \end{align*}
        A straightforward manipulation shows that for $k \geq 4$, $a = 4$, and $b = 2k-4$, $p_1(a,b) - p_2(a,b) \geq 1$. We aim for an overall lower bound of $1$. Since $p_1(a,b) - p_2(a,b)$ is clearly continuous, suppose by way of contradiction that for some $a$ and $b$, $p_1(a,b) - p_2(a,b) = 1$. By moving the terms around, this means
        \begin{align*}
            (p_1(a,b) - 1)^2 - (p_2(a,b))^2 &= 0.
        \end{align*}
        Using algebra to collect all terms on the left-hand side over a common denominator, the numerator yields the equation
        \begin{align*}
            a^4+2 a^3 (2 b^3+8 b^2+8 b+1)-a^2 (4 b^3+14 b^2+14 b+3) &\quad \\
            -2 a (2 b^3+7 b^2+7 b+2)+(b^2+3 b+2)^2 &= 0.
        \end{align*}
        Because $a = 2j \geq 4$,
        \[a^3 (2 b^3+8 b^2+8 b+1) \geq a^2 (4 b^3+14 b^2+14 b+3)\]
        and
        \[a^3 (2 b^3+8 b^2+8 b+1) \geq 2 a (2 b^3+7 b^2+7 b+2).\]
        The terms $a^4$ and $(b^2+3 b+2)^2$ are clearly positive. Overall this shows that $(p_1(a,b) - 1)^2 - (p_2(a,b))^2 > 0$, which is a contradiction. Hence, our lower bound $p_1(a,b) - p_2(a,b)$ is bounded below by $1$ and it follows that $m_j - s_{j-1} - s_j \geq 1$.
    Overall we have shown that the lower limit of each Gershgorin interval is bounded below by $1$, which is exactly what we wanted.
\end{proof}

This estimate allows us to understand the bottom of the spectrum of $\boxb^t$ on $S^3 / C_2$ and thus obtain the following conclusion.

\begin{corollary}\label{positiveSpectrum}
    With the exception of finitely many eigenvalues corresponding to the spaces $\harm_{2k}$ with $0 \leq k \leq 3$, the nonzero eigenvalues of $\boxb^t$ are bounded below by $2 h(t) = \frac{2 (1 + |t|^2)}{(1 - |t|^2)^2}$, independently of $k$. In other words, for all $|t| < 1$, 0 is not an accumulation point of the spectrum of $\boxb^t$ on $\S^3 / C_2$. It follows that for all $|t| < 1$, the quotient of the Rossi sphere by the action of $C_2$ is embeddable.
\end{corollary}

\section{Embeddability classification of other quotients of the Rossi sphere} \label{sec:rossi-embeddability-classification}

In this section, we consider other quotients of the Rossi sphere. The main result of this section is Theorem \ref{Sec4Main}, which shows that the quotient of the Rossi sphere by the action of a finite group $G\leq \on{SU}(2)$ is embeddable if and only if the order of $G$ is even. See also \cite{Epstein} for similar results.
For any vector space $V$ and finite group $G$ acting on $V$, let $V^G$ be the set of vectors in $V$ which are fixed by the action of each element of $G$. That is,
$V^G:=\{v\in V: g\cdot v=v, \,\forall g\in G\}.$
Whenever $v\in V^G$, we say that $v$ is fixed by $G$ or that $v$ is $G$-invariant.

An important fact which we use later is that if $\rho:G\to GL(V)$ is a representation of a finite group $G$ on $V$ (i.e. $g\cdot v = \rho(g)v$ for all $g\in G,v\in V$) then the standard projection onto $V^G$ is given by
$\frac{1}{\abs{G}}\sum_{g\in G}\rho(g)$.

We imitate the notation in Corollary 1.4 in  \cite{Ikeda79} and define the following spaces to relate the spectra of the Kohn Laplacian on M and on its quotients. Indeed, let $M$ be an abstract CR manifold and $\Gamma$ be a discrete group which acts smoothly, freely and properly on $M$. For any real number $\lambda$, denote the eigenspaces by  $\widehat{E}_\lambda = \{f\in L^2(M):\boxb^{M} f = \lambda f\}$ and $E_\lambda = \{f\in L^2(M/\Gamma):\boxb^{M/\Gamma} f = \lambda f\}.$ Furthermore, we use $\mathcal{A}^\Gamma$ to denote the functions in a space $\mathcal{A}$ that are invariant under the action of the group $\Gamma$.

\begin{lemma} \label{Ikedas}
    The spaces $E_\lambda$ and $\widehat{E}_\lambda^\Gamma$ are canonically isomorphic and therefore $\dim E_\lambda = \dim \widehat{E}_\lambda^\Gamma$.
\end{lemma}
\begin{proof}
    Denote by $\phi: M \to M / \Gamma$ the quotient map and let $\widehat{f} \in \widehat{E}_\lambda^\Gamma$. As $\widehat{f}$ is fixed by $\Gamma$, it canonically defines a function $f \in L^2(M / \Gamma)$ given by $f(p) = \widehat{f}(q)$ where $q \in \phi^{-1}(\{p\})$ arbitrarily. Furthermore, $f \in E_\lambda$ because for any $q \in M$ we can take an open neighborhood $U$ of $q$ small enough to contain at most one point in each orbit of $\Gamma$ and consider the action of $\boxb^{M / \Gamma}$ on $\phi(U)$; we apply Proposition \ref{kohn-laplacian-commute} on $U$ which shows that $f$ is an eigenfunction of $\boxb^{M / \Gamma}$ with eigenvalue $\lambda$. Conversely, given any $f \in E_\lambda$, there is canonically a corresponding $\widehat{f} \in L^2(M)$ given by $\widehat{f}(q) = f(\phi(q))$, which is fixed by $\Gamma$ by construction. By a similar argument as before, $\widehat{f}$ actually is in $\widehat{E}_\lambda^\Gamma$. Overall we have described a bijection $E_\lambda \to \widehat{E}_\lambda^\Gamma$; this bijection is clearly linear, so it is an isomorphism. The equality of dimensions immediately follows: $\dim E_\lambda = \dim \widehat{E}_\lambda^\Gamma$.
\end{proof}

At this point we are able to conclude that quotients of the Rossi sphere by even order finite subgroups of $\on{SU}(2)$ are embeddable. We require this statement later in the proof of the main result of the section (Theorem \ref{Sec4Main}).

\begin{proposition}\label{evens}
    If $\Gamma$ is a discrete group of even order then $\boxb^t$ has positive spectrum on $\S^3/\Gamma$. 
\end{proposition}

\begin{proof}
    If $\Gamma$ has even order, it must have a subgroup isomorphic to $C_2$. Because $\on{SU}(2)$ contains a unique involution, $\Gamma $ must contain the particular subgroup $C_2=\{(1,0),(-1,0)\}$.
     Since $C_2 \leq \Gamma$, any $\Gamma$-invariant function on $\S^3$ is also $C_2$-invariant; this shows that the spectrum of $\boxb^t$ on $\S^3 / \Gamma$ is a subset of the spectrum of $\boxb^t$ on $\S^3 / C_2$. By Corollary \ref{positiveSpectrum}, the latter set does not have an accumulation point at $0$, so neither does the former.
\end{proof}

In order to study $\boxb^t$ on quotients by subgroups of $\on{SU}(2)$, we first need to establish basic facts about how $\mathcal{L}$ and $\overline{\mathcal{L}}$ interact with the action of $\on{SU}(2)$.

\begin{lemma}\label{redbox}
    For any $g\in (2)$ and $f\in L^2(\S^3)$,
    \[\mathcal{L}(\ell_g^* f)=\ell_g^* (\mathcal{L} f) \text{, }\overline{\mathcal{L}}(\ell_g^* f)=\ell_g^* (\overline{\mathcal{L}} f)
    \text{, and } \boxb^t(\ell_g^* f) = \ell_g^*(\boxb^t).  \]
\end{lemma}
\begin{proof}
     Recall that $\Ell, \bar{\Ell}$ are left-invariant complexified vector fields. We have that
     \[\ell_g^* (\Ell f) = (\Ell f) \circ \ell_g,\]
     but also 
     \[\Ell (\ell_g^* f) = \Ell (f \circ \ell_g) = ((\ell_{g^*} \Ell) f) \circ \ell_g = (\Ell f) \circ \ell_g,\]
     where we have used Corollary 8.21 from \cite{lee-smooth-manifolds} and the left-invariance of $\Ell$. This completes the proof for $\Ell$, and the proof is entirely the same for $\bar{\Ell}$.  Since $\boxb^t$ is a linear combination of compositions of $\Ell$ and $\bar{\Ell}$, it follows that $\boxb^t(\ell_g^* f) = \ell_g^*(\boxb^t) $ as well.
\end{proof}

Now, we consider the action of subgroups of $\on{SU}(2)$ on the sphere and note that the isomorphic subgroups lead into the same spectrum.

\begin{theorem}\label{isospect}
For $|t|<1$, the spectrum of the Kohn Laplacian $\boxb^t$ is the same on quotients of the sphere obtained by two isomorphic subgroups of $\on{SU}(2)$. 
\end{theorem}
\begin{proof}
Let $G_1,G_2\leq \on{SU}(2)$ be two isomorphic groups. Then by \cite[Chapter 1, Section 6]{Zass}, $G_2=\tau G_1 \tau^{-1}$ for some $\tau\in\on{SU}(2)$. It is not hard to verify that $\ell_\tau^*$ gives an isomorphism between $\widehat{E}_\lambda^{G_1}$ and $\widehat{E}_\lambda^{G_2}$ so $\dim(\widehat{E}_\lambda^{G_1})=\dim(\widehat{E}_\lambda^{G_2})$ and the theorem follows.
\end{proof}

Returning to the context of the proof of Theorem \ref{isospect} we can prove a less flashy but arguably more powerful lemma.

\begin{lemma}\label{harmony}
    Isomorphic subgroups $G_1\cong G_2$ of $\on{SU}(2)$ fix isomorphic subspaces of $\harm_{0,k}$ for each $k$. That is, $\dim(\harm_{0,k}^{G_1})=\dim(\harm_{0,k}^{G_2}).$
\end{lemma}
\begin{proof}

    In the proof of Theorem \ref{isospect} we noted that there exists some $\tau \in \on{SU}(2)$ such that \[G_2 = \{\tau \cdot g \cdot \tau^{-1}: g\in G_1\}\] and that $\ell_\tau^*$ is an invertible linear transformation that maps eigenfunctions of $G_1$ to eigenfunctions of $G_2$. In particular, $\ell_\tau^*$ maps functions fixed by $G_1$ to functions fixed by $G_2$. This theorem follows by further verifying that $\ell_\tau^*$ preserves both bi-degree and the property of being harmonic, so it provides an isomorphism between $\harm_{0,k}^{G_1}$ and $\harm_{0,k}^{G_2}.$
\end{proof}

Lemma \ref{harmony} can be used to prove Theorem \ref{isospect}, but the argument is more indirect and requires the machinery of Proposition \ref{Gamma decomp}. We need that machinery to conclude that quotients by odd order subgroups are not embeddable so we move towards proving it with the next two lemmas.

\begin{lemma}\label{LGamma}
    If a function $f$ is fixed by $\Gamma$, then $\overline{\mathcal{L}}^m f$ and $\mathcal{L}^m f$ are also fixed by $\Gamma$ for any $m\geq 1$.
\end{lemma}
\begin{proof}
     A function $f$ is fixed by $\Gamma$ iff $\ell_g^* f = f$ for all $g\in \Gamma$. 
     By Lemma \ref{redbox}, $\ell_g^* (\overline{\mathcal{L}} f) = \overline{\mathcal{L}}(\ell_g^* f) = \overline{\mathcal{L}}f $ for all $g\in \Gamma$. Hence, if $f$ is fixed by $\Gamma$ then $\overline{\mathcal{L}}f$ is also fixed by $\Gamma$. It follows by induction that if $f$ is fixed by $\Gamma$ then $\overline{\mathcal{L}}^m f$ is fixed by $\Gamma$ for all $m$. The proof that $\mathcal{L}^m f$ is fixed by $\Gamma$ as well can be obtained by replacing each instance of $\overline{\mathcal{L}}$ with $\mathcal{L}$.
\end{proof}

\begin{lemma}\label{Vinvar}
    Let $f\in \mathcal{H}_{0,k}$ be fixed by $\Gamma$. Every vector in $V_f\oplus W_f = \{f,\overline{\mathcal{L}}f,\overline{\mathcal{L}}^2 f, \overline{\mathcal{L}}^3 f,\dots \}$, is also fixed by $\Gamma$.
\end{lemma}
\begin{proof}
     This follows directly from Lemma \ref{LGamma}.
\end{proof}

Next, we relate the invariant functions and the subspaces $V_f$ and $W_f$ more precisely. Recall that these subspaces played an important role in representing the Kohn Laplacian as matrices.

\begin{proposition}\label{Gamma decomp}
Let \(f_1,f_2, \dots, f_m\) be an orthonormal basis for $(\mathcal{H}_{0,k})^\Gamma$. Then $(\mathcal{H}_k)^\Gamma = \bigoplus_{i=1}^m V_{f_i}\oplus W_{f_i}$.
\end{proposition}
\begin{proof}
    Let \(f_1,f_2,\dots, f_m, f_{m+1}, \dots, f_{k+1}\) be the completion of the given basis to an orthonormal basis for all of $\mathcal{H}_{0,k}$. The fact that $\mathcal{H}_k = \bigoplus_{i=1}^{k+1} V_{f_i}\oplus W_{f_i}$ together with Lemma \ref{Vinvar} demonstrates that $\bigoplus_{i=1}^m V_{f_i}\oplus W_{f_i} \subseteq (\mathcal{H}_k)^\Gamma$; thus we just need to show that $(\mathcal{H}_k)^\Gamma \subseteq \bigoplus_{i=1}^m V_{f_i}\oplus W_{f_i}.$
    
    Suppose towards a contradiction that there exists some $f\in \mathcal{H}_k^\Gamma \setminus \bigoplus_{i=1}^m V_{f_i}\oplus W_{f_i}.$ 
    We can write $f=\sum_{i=1}^{k+1} g_i$ where $g_i\in V_{f_i}\oplus W_{f_i}$. Thus, $f-\sum_{i=1}^m g_i = \sum_{i=m+1}^{k+1} g_i$. Both $f$ and $\sum_{i=1}^m g_i$ are fixed by $\Gamma$ so $f-\sum_{i=1}^m g_i$ and hence $\sum_{i=m+1}^{k+1} g_i$ is also fixed by $\Gamma$.
    
    There is some $\eta$ such that $\mathcal{L}^\eta \sum_{i=m+1}^{k+1} g_i = 0$ but $\mathcal{L}^{\eta-1}\sum_{i=m+1}^{k+1} g_i\neq 0$. The kernel of $\mathcal{L}$ on $\mathcal{H}_k$ is $\mathcal{H}_{0,k}$ so $\mathcal{L}^{\eta-1}\sum_{i=m+1}^{k+1} g_i \in \mathcal{H}_{0,k}$. Since $\sum_{i=m+1}^{k+1} g_i$ is fixed by $\Gamma$, $\mathcal{L}^{\eta-1}\sum_{i=m+1}^{k+1} g_i$ is also fixed by $\Gamma$ and hence $\mathcal{L}^{\eta-1}\sum_{i=m+1}^{k+1} g_i \in (\mathcal{H}_{0,k})^\Gamma$.
    
    Since $\mathcal{L} (V_{f_i} \oplus W_{f_i}) \subseteq V_{f_i} \oplus W_{f_i}$ for any $f$,  $\mathcal{L}^{\eta-1} (V_{f_i}\oplus W_{f_i}) \subseteq  V_{f_i}\oplus W_{f_i}$. Hence, \[\mathcal{L}^{\eta-1} \left( \sum_{i=m+1}^{k+1} g_i \right) \in \bigoplus_{i=m+1}^{k+1} V_{f_i}\oplus W_{f_i}.\]
    Because $f_{m+1},\dots, f_{k+1}$ gives an orthogonal basis for $\mathcal{H}_{0,k} \cap \bigoplus_{i=m+1}^{k+1} V_{f_i}\oplus W_{f_i}$, we have 
    $\mathcal{L}^{\eta-1}\sum_{i=m+1}^{k+1} g_i \in ((\mathcal{H}_{0,k})^\Gamma)^\perp$. This contradicts our earlier finding.
\end{proof}

The consequence of Proposition \ref{Gamma decomp} is that the multiplicity of an eigenvalue in $(\mathcal{H}_k)^\Gamma$ is proportional to the dimension of $(\mathcal{H}_{0,k})^\Gamma$. This motivates us to study the spaces $(\mathcal{H}_{0,k})^\Gamma$ and distributions of their corresponding eigenvalues. All odd-order subgroups of $\on{SU}(2)$ are cyclic groups (see \cite[Chapter 1, Section 6]{Zass}). In order to attack the question of embeddability in this case, we examine which spherical harmonics are invariant under these group actions. Note that in Section \ref{three}, we constructed $V_f$ and $W_f$ for $f \in \harm_{0,2k}$. An analogous construction holds for $f \in \harm_{0,2k+1}$ (see \cite{REU17}).

\begin{lemma}\label{attempt2}
The action of any subgroup of $\on{SU}(2)$ isomorphic to $C_{2m+1}$ leaves some $f\in \mathcal{H}_{0,2k+1}$ fixed for each $k\geq m$. 
\end{lemma}
\begin{proof}
Let $\xi$ be a root of unity with order $2m+1$. $(\xi,0)$ generates a subgroup of $\S^3$ isomorphic to $C_{2m+1}$. Consider $f(z_1,z_2)=\overline{z_1}^{k+m+1}\overline{z_2}^{k-m}$, which is an element of $\mathcal{H}_{0,2k+1}$ as long as $k\geq m$. A direct manipulation shows that
$f((\xi, 0)\cdot (z_1,z_2)) = f(\xi z_1,\overline{\xi}z_2) = \overline{\xi z_1}^{k+m+1}\overline{\overline{\xi}z_2}^{k-m} =\xi^{-(k+m+1)}\xi^{ k-m}\overline{z_1}^{k+m+1}\overline{z_2}^{k-m} = \xi^{-(2m+1)}f(z_1,z_2) = f(z_1,z_2).$
Therefore, $(\xi,0)\cdot f = f$, that is, $f$ is invariant under the action of this subgroup. This proves that the action of the cyclic group $\langle (\xi,0)\rangle $ leaves some $f\in \mathcal{H}_{0,2k+1}$ fixed for each $k\geq m$. By Lemma \ref{harmony}, it follows that for any $G\leq \on{SU}(2)$ with $G\cong C_{2m+1}$, $\harm_{0,2k+1}^G$ is nonempty.
\end{proof}

We can now use the $f$ from Lemma \ref{attempt2} to construct the subspace $W_f$.

\begin{lemma}\label{attempt1}
For all $k\geq m$, there exists an $f\in \mathcal{H}_{0,2k+1}$ such that the action of any subgroup of $\on{SU}(2)$ that is isomorphic to $C_{2m+1}$ leaves everything fixed in $W_{f}\subseteq \mathcal{H}_{2k+1}$. 
\end{lemma}
\begin{proof}
Let $k \geq m$ be arbitrary. By Lemma \ref{attempt2}, there exists some $f \in \harm_{0,2k+1}$ which is invariant under the action of the group $C_{2m+1}$. Thus, by Lemma \ref{Vinvar}, all functions in $V_f \oplus W_f$ space $\{f, \bar{\Ell} f, \dots, \bar{\Ell}^{2k+1}\}$ are $C_{2m+1}$-invariant.
\end{proof}
Finally, by combining all these results we can conclude the main result of this section.

\begin{theorem}\label{Sec4Main}
   If $\Gamma$ is a discrete subgroup of $\on{SU}(2)$ then  $(\S^3/\Gamma,\mathcal{L}_t)$ is embeddable if and only if $\abs{\Gamma}$ is even.
\end{theorem}
\begin{proof}
If $\abs{\Gamma}$ is even, then $C_2 \leq \Gamma$ and thus $\boxb^t$ has positive spectrum on $\S^3/\Gamma$ by Corollary \ref{evens}. If $\abs{\Gamma}$ is odd then $\Gamma$ must be a cyclic group. In this case, with the exception of finitely many small $k$
, Lemma \ref{attempt1} guarantees that the spectrum of $\boxb^t$ on the quotient contains all the eigenvalues of $\boxb^t$ on the full Rossi sphere which arise from $\harm_{2k+1}$ spaces. It was noted in \cite{REU17} that the eigenvalues for the full Rossi sphere arising from the space $\bigoplus_{k=0}^\infty \harm_{2k+1}$ have an accumulation point at $0$, so in this case the quotient has arbitrarly small eigenvalues. Recalling that positivity of the spectrum is equivalent to embeddability, we obtain the desired conclusion.
\end{proof}

\section{Hearing the size of the group on the quotients of $\S^3$} \label{sec:hear-group-order}

In this section, we consider the quotients of the sphere with the standard CR structure. We relate the spectrum on $\S^3/\Gamma$ and the order of $\Gamma.$
Recall from Proposition \ref{action-l-lbar} that the eigenvalue of the (unperturbed) Kohn Laplacian for $\S^3$ acting on $\harm_{p,q}$ is $2 q (p+1)$. 
We look at the dimensions of invariant functions under a group action.

\begin{lemma}\label{cyclic case} Let $\Gamma$ be a cyclic subgroup of $\S^3$ and $d=|\Gamma|$.\begin{itemize}
    \item If $d$ is even then $\dim (\mathcal{H}_{0,k}^\Gamma) = \begin{cases} 2\lfloor \tfrac{k}{d}\rfloor +1 & 2\mid k \\ 0 & 2\nmid k \end{cases}.$
    
    \item If $d$ is odd then $\dim (\mathcal{H}_{0,k}^\Gamma) = 2\lfloor \tfrac{k}{2d}\rfloor+(1-(-1)^k)/2 .$
\end{itemize}

\end{lemma}
\begin{proof}
     By Lemma \ref{harmony}, we only need to consider $\Gamma$ generated by $(\xi,0)$ with $\xi$ a $d$-th root of unity.
     Since $\Gamma$ is generated by $(\xi,0)$, a function $f:\S^3\to \C$ is fixed by $\Gamma$ if and only if $f((\xi,0)\cdot (z_1,z_2))  = f(z_1,z_2)$ for all $(z_1,z_2)\in \S^3$.
    
    The eigenspace of $\ell_{(\xi,0)}^*$ associated with the eigenvalue 1 is $(\mathcal{H}_{0,k})^\Gamma$; we are interested in the dimension of this eigenspace. Let $f_a(z_1,z_2)=\overline{z_1}^a\overline{z_2}^{k-a}$. The functions $f_0,\dots, f_{k}$ are an orthogonal basis for $\mathcal{H}_{0,k}$. Furthermore,
    \[f_a((\xi,0)\cdot (z_1,z_2)) = f_a(\xi z_1,\bar{\xi} z_2) =\xi^a z_1^a \xi^{-(k-a)}z_2^{k-a} = \xi^{2a-k}f_a.\]
    Therefore, $f_a$ is an eigenfunction of $\ell_{(\xi,0)}^*$ associated with the eigenvalue $\xi^{2a-k}$, and so $\{f_a:\xi^{2a-k}=1\}$ is an orthogonal basis for $(\mathcal{H}_{0,k})^\Gamma$. Recalling that $\xi$ is a primitive $d$-th root of unity, we observe that $\dim (\mathcal{H}_{0,k})^\Gamma = \#\{a\in [0,k]: d\mid 2a-k\}$, where  $\#A$ denotes the cardinality of the set $A$. This quantity depends on the parities of $d$ and $k$:
    \begin{itemize}
        \item If $d$ is even and $k$ is odd, $\dim (\mathcal{H}_{0,k})^\Gamma = 0$.
        \item If $d$ is even and $k$ is even, then we compare the number of elements in the following sets \[\#\{a: d\mid 2a-k\}= \#\{r\in [-k/2,k/2]:d/2\mid r\}=\#\{0,\pm d/2,\pm d, \pm 3d/2,\dots, \pm \left \lfloor \tfrac{k/2}{d/2} \right\rfloor d/2\}\] and so $\dim (\mathcal{H}_{0,k})^\Gamma = 2\lfloor \frac{k}{d}\rfloor +1$.
        \item If $d$ is odd and $k$ is even, then \[\dim (\mathcal{H}_{0,k})^\Gamma = \#\{a: d\mid 2a-k\}= \#\{r\in [-k/2,k/2]:d \mid r\} = 2\lfloor \tfrac{k}{2d}\rfloor+1.\]
        \item If $d$ is odd and $k$ is odd, then \[\dim (\mathcal{H}_{0,k})^\Gamma = \#\{a: d\mid 2a-k\}= \#\{r\in [-k/2,k/2]:d \mid 2r+1\} = 2\lfloor \tfrac{k}{2d}\rfloor.\]
    \end{itemize}
\end{proof}

\begin{lemma} \label{lem:hear-odd-order}
    Let $\Gamma$ be a subgroup of $\on{SU}(2)$ of odd order. Given the spectrum of $\boxb$ on $\S^3 / \Gamma$, one can hear the order of $\Gamma$.
\end{lemma}
\begin{proof}
    Since we know $\Gamma$ is an odd-order subgroup of $\on{SU}(2)$ it must be cyclic, so we may refer to Lemma \ref{cyclic case}. Let $\alpha$ be a prime greater than $2$ (viz. odd), and consider the dimension of the eigenspace $E_{2\alpha}$ of $\boxb$ on the quotient corresponding to the eigenvalue $2\alpha$. Since the eigenvalue of $\boxb$ coming from $\harm_{p,q}$ is $2q(p+1)$, the eigenvalue $2\alpha$ may come from two spherical harmonic spaces: $\harm_{0,\alpha}$ and $\harm_{\alpha-1,1}$. The dimension of this eigenspace is hence
    \[\dim(E_{2\alpha}) = 2 \dim(\harm_{0,\alpha}^\Gamma).\]
    The quantity $\dim(E_{2\alpha})$ is audible for each prime greater than $2$, and consequently the limit
    \[\lim_{\substack{\alpha \to \infty \\ \alpha \text{ prime}}} \frac{\dim(E_{2\alpha})}{\alpha}\]
    is audible too. We know from Lemma \ref{cyclic case} the exact dimension of this eigenspace, which is $4 \lfloor \frac{\alpha}{2 |\Gamma|} \rfloor + 2$ (here we used the fact that $\alpha$ was odd), and we can compute this limit directly. It is
    \[\lim_{\substack{\alpha \to \infty \\ \alpha \text{ prime}}} \frac{\dim(E_{2\alpha})}{\alpha} = \lim_{\alpha} \frac{4 \lfloor \frac{\alpha}{2 |\Gamma|} \rfloor + 2}{\alpha} = \lim_{\alpha} \frac{4 \lfloor \frac{\alpha}{2 |\Gamma|} \rfloor}{\alpha} = \frac{2}{|\Gamma|}.\]
    From this we conclude that one can hear the order of $\Gamma$ in the case where the order is already known to be odd.
\end{proof}

As noted earlier, all odd-order subgroups of $\on{SU}(2)$ are cyclic; therefore, the preceding lemma completely covers the odd-order case. However, the even-order subgroups of $\on{SU}(2)$ are more varied, and Lemma \ref{cyclic case} does not directly tell us anything if $\Gamma$ is an even-order non-cyclic subgroup. Thus, in order to hear the order of all even subgroups, we need the following pieces.

\begin{lemma}\label{lim2}
If $\Gamma\leq \on{SU}(2)$ is an even order cyclic subgroup and $\chi_k$ is the character corresponding to the representation $\rho_k:\on{SU}(2)\to \textup{Aut}(\mathcal{H}_{0,k})$ given by $g\mapsto \ell_g^*$, then
\[\lim_{k\to \infty} \frac{1}{2k}\sum_{g\in \Gamma}\chi_{2k}(g) = 2\]
 \end{lemma}
 \begin{proof}
 By basic character theory, $\frac{1}{\abs{\Gamma}}\sum_{g\in \Gamma}\chi_k(g) =\dim(\mathcal{H}_{0,k}^\Gamma).$
 By Lemma \ref{cyclic case}, if $\Gamma$ is a group of order $d=2d'$ then $\dim(\mathcal{H}_{0,2k}^\Gamma) = 2 \left\lfloor \frac{2k}{d} \right\rfloor +1.$
Thus, 
$\frac{1}{d}\sum_{g\in \Gamma}\chi_{2k}(g) = 2 \left\lfloor \frac{2k}{d} \right\rfloor +1.$
Rearranging this gives
$\frac{1}{2k}\sum_{g\in \Gamma}\chi_{2k}(g) =\frac{d}{2k}\left( 2 \left\lfloor \frac{2k}{d} \right\rfloor +1 \right).$
And therefore,
\[\lim_{k\to \infty}\frac{1}{2k}\sum_{g\in \Gamma}\chi_{2k}(g) = 2.\]
\end{proof}

Lemma \ref{lim2} only tells us about the characters of even order cyclic groups. The following proposition expresses a group in terms of its cyclic subgroups so that we can apply Lemma \ref{lim2} to arbitrary even order subgroups of $\on{SU}(2)$.

\begin{proposition}
Every even order finite subgroup of $\on{SU}(2)$ can be written as the union of even cyclic subgroups.
\end{proposition}
\begin{proof}
We claim that if $G\leq \on{SU}(2)$ is a finite group, then $G=\bigcup_{g\in G, 2\mid o(g)}\langle g \rangle.$
Clearly every even order element and every even order cyclic subgroup is contained in this union. We just need to show that every odd order element is contained in this union. The trick here is that $\on{SU}(2)$ has a unique element of order 2 and it is central so every odd order element $g$ in $G$ is contained in an even order cyclic subgroup of $G$ generated by $-1\cdot g$ where $-1$ denotes the unique involution.
\end{proof}

Understanding the spectrum of the Kohn Laplacian on the quotient of $\S^3$ by some $\Gamma\leq \on{SU}(2)$ comes down to understanding which spherical harmonics are fixed by the action of $\Gamma$. In Theorem \ref{thm:even-quotients-asymptotics} we use relatively elementary means to get the asymptotic information we need about the dimensions of spaces of spherical harmonics fixed by $\Gamma$. However, another approach would have been to use  \cite[Theorem 4]{pnas} to show the same result. The results presented in \cite{pnas} are significantly stronger and more technical than what we need here, so we have opted not to include them.

\begin{theorem} \label{thm:even-quotients-asymptotics}
    If $\Gamma$ is any even order finite subgroup of $\on{SU}(2)$, then $\lim_{k\to \infty} \frac{\dim (\mathcal{H}_{0,2k})^\Gamma}{2k+1}=\frac{1}{\abs{\Gamma}}$.
\end{theorem}
\begin{proof}
     Let $\Gamma \leq \on{SU}(2)$ be a finite subgroup of $\on{SU}(2)$. Let $C_1\dots C_s$ be even-order cyclic subgroups of $\Gamma$ such that $\Gamma =\bigcup_{i\leq s} C_i$.
     Let $\chi_k$ be the character of the representation $\rho:\on{SU}(2)\to Aut(\mathcal{H}_{0,k})$ given by $g\mapsto \ell_g^*$.
     The principle of inclusion-exclusion gives
   \begin{equation*}
        \sum_{g\in \Gamma}\chi_k(g) = \sum_{r=1}^s \left[ -(-1)^{r} \sum_{1\leq i_1<i_2<\dots<i_r\leq s} \left[ \sum_{g\in C_{i_1}\cap C_{i_2} \cap \dots C_{i_r}} \chi_k(g) \right] \right]
   \end{equation*}
  
   The intersection of cyclic groups is cyclic, and the intersection of even cyclic subgroups of $\on{SU}(2)$ contains the unique involution and thus has even order.  We have shown for any even order cyclic group $C\leq \on{SU}(2)$ that $\lim_{k\to \infty}\frac{1}{k}\sum_{g\in C}\chi_k(g)=2$. Therefore, by taking the limit of the inclusion-exclusion identity above divided by $k$, we obtain
   \begin{align*}
        \lim_{k\to \infty}\frac{1}{k}\sum_{g\in \Gamma}\chi_k(g) &= \sum_{r=1}^s -(-1)^{r}\sum_{1\leq i_1<i_2<\dots<i_r\leq s} 
         \lim_{k\to \infty}\frac{1}{k}\sum_{g\in C_{i_1}\cap C_{i_2} \cap \dots C_{i_r}} \chi_k(g)\\
         &=\sum_{r=1}^s -(-1)^{r}\sum_{1\leq i_1<i_2<\dots<i_r\leq s} 
         2
         =2\sum_{r=1}^s -(-1)^{r} \binom{s}{r}
         =2\binom{s}{0}
         =2
   \end{align*}
   Since $\frac{1}{\abs{\Gamma}}\sum_{g\in \Gamma}\chi_k(g) = \dim (\mathcal{H}_{0,2k})^\Gamma$, it follows that \[\lim_{k\to \infty} \frac{\dim (\mathcal{H}_{0,2k})^\Gamma}{2k+1}=\frac{1}{\abs{\Gamma}}\]
\end{proof}

\begin{corollary}
    If $\Gamma$ is any even order finite subgroup of $\on{SU}(2)$, then for any $p+q = 2k$, \[\lim_{k\to \infty} \frac{\dim (\mathcal{H}_{p,q})^\Gamma}{2k+1} = \lim_{k\to \infty} \frac{\dim (\mathcal{H}_{2k})^\Gamma}{(2k+1)^2} = \frac{1}{\abs{\Gamma}}.\]
\end{corollary}
\begin{proof}
     By Lemma \ref{Vinvar}, one may repeatedly apply $\bar{\Ell}$ to any element of $(\harm_{0,2k})^\Gamma$ to generate a $V \oplus W$ space of $\Gamma$-invariant functions. Each such $V \oplus W$ space contains exactly one 1-dimensional subspace contained in each $\harm_{p,q} \subset \harm_{2k}$ which shows that the asymptotic dimension of $(\harm_{0,2k})^\Gamma$ extends to $(\harm_{p,q})^\Gamma$ and consequently $(\harm_{2k})^\Gamma$.
\end{proof}

\begin{lemma} \label{lem:hear-even-order}
    Let $\Gamma$ be a subgroup of $\on{SU}(2)$ of even order. Given the spectrum of $\boxb$ on $\S^3 / \Gamma$, one can hear the order of $\Gamma$.
\end{lemma}
\begin{proof}
    Let $\alpha$ be a prime greater than $2$ (viz. odd), and consider the dimension of the eigenspace $E_{4\alpha}$ of $\boxb$ on the quotient corresponding to the eigenvalue $4\alpha$. Since the eigenvalue of $\boxb$ coming from $H_{p,q}$ is $2q(p+1)$, the eigenvalue $4\alpha$ may come from four spherical harmonic spaces: $\harm_{0,2\alpha}^\Gamma$, $\harm_{2\alpha-1, 1}^\Gamma$, $\harm_{1,\alpha}^\Gamma$, and $\harm_{\alpha-1,2}^\Gamma$. The dimension of this eigenspace is hence
    \[\dim(E_{4\alpha}) = \dim(\harm_{0,2\alpha}^\Gamma) + \dim(\harm_{2\alpha-1, 1}^\Gamma) + \dim(\harm_{1,\alpha}^\Gamma) + \dim(\harm_{\alpha-1,2}^\Gamma) = 2 (\dim(\harm_{0,2\alpha}^\Gamma) + \dim(\harm_{0,\alpha+1}^\Gamma)).\]
    The quantity $\dim(E_{4\alpha})$ is audible for each prime greater than $2$, and consequently the limit
    \[\lim_{\substack{\alpha \to \infty \\ \alpha \text{ prime}}} \frac{\dim(E_{4\alpha})}{\alpha}\]
    is audible too. Thanks to Theorem \ref{thm:even-quotients-asymptotics}, the value of this limit is
    \begin{align*}
        \lim_{\substack{\alpha \to \infty \\ \alpha \text{ prime}}} \frac{\dim(E_{4\alpha})}{\alpha} &= 2 \lim_{\alpha} \frac{\dim(\harm_{0,2\alpha}^\Gamma) + \dim(\harm_{0,\alpha+1}^\Gamma)}{\alpha} = 2 \left( \lim_{\alpha} \frac{\dim(\harm_{0,2\alpha}^\Gamma)}{\alpha} + \lim_{\alpha} \frac{\dim(\harm_{0,\alpha+1}^\Gamma)}{\alpha} \right) \\
        &= 2 \left( 2 \lim_{\alpha} \frac{\dim(\harm_{0,2\alpha}^\Gamma)}{2\alpha + 1} + \lim_{\alpha} \frac{\dim(\harm_{0,\alpha+1}^\Gamma)}{\alpha + 2} \right) = 2 \left( \frac{2}{|\Gamma|} + \frac{1}{|\Gamma|} \right) = \frac{6}{|\Gamma|}.
    \end{align*}
    Therefore, one can hear the order of $\Gamma$.
\end{proof}

We now have all the ingredients to prove the main conclusion of this section.

\begin{theorem}
    Let $\Gamma$ be a finite subgroup of $\on{SU}(2)$. Given the spectrum of $\boxb$ on $\S^3 / \Gamma$, one can hear the order of $\Gamma$.
\end{theorem}
\begin{proof}
    The first step is to identify the parity of $|\Gamma|$. To this end, pick your favorite odd prime $\alpha$ and check whether the spectrum of $\boxb$ contains the eigenvalue $2\alpha$. That eigenvalue is in the spectrum if and only if $|\Gamma|$ is odd, thanks to Lemma \ref{cyclic case}. Hence, we can identify whether $\Gamma$ has even or odd order, and the proof is completed by invoking either Lemma \ref{lem:hear-even-order} or Lemma \ref{lem:hear-odd-order} according to the parity of $|\Gamma|$.
\end{proof}

\section{Further Directions}

In this section, we note a few directions for further investigation following the work in this note. Although we have some partial results for some of these directions, we kept them out of this note for further development.

\begin{itemize}
\item In this note, we considered quotients of the sphere by finite subgroups of $\on{SU}(2)$. How much of this work also holds for finite the subgroups of $\on{SO}(4)$ (i.e. more general classes of spherical 3-manifolds)?
\item In Section \ref{sec:hear-group-order}, we showed that one can hear the order of the quotienting group $\Gamma$. Can one hear the exact group $\Gamma$?
\item Can we hear the order of the group $\Gamma$ when we look at the quotients of the Rossi sphere? Note that in this case, we do not have an explicit formula for the eigenvalues and the spectrum has accumulation points.
\item In Section \ref{sec:hear-group-order}, we only looked at the 3-sphere $\S^3$, but we know the exact eigenvalues of $\boxb$ on other spheres $\S^{2n-1}$ too. Can one look at the Kohn Laplacian on the quotients of higher dimensional spheres?
\end{itemize}

\section*{Acknowledgements} 
We would like to thank the referees for constructive feedback.
This research was conducted at the NSF REU Site (DMS-1950102, DMS-1659203) in Mathematical Analysis and Applications at the University of Michigan-Dearborn. We would like to thank the National Science Foundation, National Security Agency, and University of Michigan-Dearborn for their support. 

\bibliographystyle{alpha}

\newcommand{\etalchar}[1]{$^{#1}$}

\end{document}